\pdfoutput=1
\documentclass[letterpaper, 10pt, conference]{ieeeconf}
\usepackage{xcolor}
\usepackage[english]{babel}
\usepackage{amsmath} 
\usepackage{amssymb} 
\usepackage{amsthm}

\usepackage{graphicx} 
\usepackage{multirow}
\usepackage{booktabs}

\definecolor{matlabblue}{rgb}{0,0.4470,0.7410}
\definecolor{matlabred}{rgb}{0.8500,0.3250,0.0980}
\definecolor{matlabyellow}{rgb}{0.9290,0.6940,0.1250}
\definecolor{matlabpurple}{rgb}{0.4940,0.1840,0.5560}
\definecolor{matlabgreen}{rgb}{0.4660,0.6740,0.1880}
\definecolor{matlabblack}{rgb}{0,0,0}
\definecolor{matlabmagenta}{rgb}{1,0,1}

\newcommand{\red}[1]{{\color{red}#1}}

\renewcommand{\vec}[1]{\mathbf{#1}}
\newcommand{\ddt}[1]{\frac{{\rm d}#1}{{\rm d} t}}
\newcommand{\smallddt}{\tfrac{{\rm d}}{{\rm d} t}}
\newcommand{\rd}{{\rm d}}
\newcommand{\E}{\mathbb{E}}
\newcommand{\Prob}{{\rm Pr}}
\newcommand{\objFun}{g}
\newcommand{\conFun}{\vec{h}}

\definecolor{dp1}{rgb}{0.1094,0.1094,0.1927}
\definecolor{dp2}{rgb}{0.2188,0.2188,0.3438}
\definecolor{dp3}{rgb}{0.3281,0.3698,0.4531}
\definecolor{dp4}{rgb}{0.4375,0.5208,0.5625}
\definecolor{dp5}{rgb}{0.5469,0.6719,0.6719}
\definecolor{dp6}{rgb}{0.6979,0.7812,0.7812}
\newcommand{\dpone}[1]{{\color{dp1}#1}}
\newcommand{\dptwo}[1]{{\color{dp2}#1}}
\newcommand{\dpthree}[1]{{\color{dp3}#1}}
\newcommand{\dpfour}[1]{{\color{dp4}#1}}
\newcommand{\dpfive}[1]{{\color{dp5}#1}}
\newcommand{\dpsix}[1]{{\color{dp6}#1}}

\usepackage{mathtools}

\PassOptionsToPackage{hyphens}{url}
\usepackage[colorlinks=true,linkcolor=blue,urlcolor=blue,citecolor=blue]{hyperref}
\usepackage[nameinlink]{cleveref}
\crefformat{equation}{(#2#1#3)}
\crefrangeformat{equation}{(#3#1#4)--(#5#2#6)}
\crefmultiformat{equation}{(#2#1#3)}{ and~(#2#1#3)}{, (#2#1#3)}{ and~(#2#1#3)}

\newtheorem{theorem}{Theorem}
\newtheorem{proposition}{Proposition}
\newtheorem{lemma}{Lemma}
\theoremstyle{remark}

\IEEEoverridecommandlockouts 
\title{\LARGE \bf
Uncertainty propagation for nonlinear dynamics: \\A polynomial optimization approach}

\author{Francesca Covella and Giovanni Fantuzzi
\thanks{F. Covella (\href{mailto:francesca.covella@polimi.it}{francesca.covella@polimi.it}) is with the Department of Aerospace Science and Technology at Politecnico di Milano.}
\thanks{G. Fantuzzi (\href{mailto:giovanni.fantuzzi@fau.de}{giovanni.fantuzzi@fau.de}) is with the Department of Data Science, FAU Erlangen--N\"urnberg. Part of this work was supported by an Imperial College Research Fellowship.}
\thanks{For the purpose of open access, the authors have applied a CC-BY licence to any Author Accepted Manuscript version arising.}
}

\begin{document}
\maketitle
\begin{abstract}
We use Lyapunov-like functions and convex optimization to propagate uncertainty in the initial condition of nonlinear systems governed by ordinary differential equations. 
We consider the full nonlinear dynamics without approximation, producing rigorous bounds on the expected future value of a quantity of interest even when only limited statistics of the initial condition (e.g., mean and variance) are known.
For dynamical systems evolving in compact sets, the best upper (lower) bound coincides with the largest (smallest) expectation among all initial state distributions consistent with the known statistics. For systems governed by polynomial equations and polynomial quantities of interest, one-sided estimates on the optimal bounds can be computed using tools from polynomial optimization and semidefinite programming. Moreover, these numerical bounds provably converge to the optimal ones in the compact case.
We illustrate the approach on a van der Pol oscillator and on the Lorenz system in the chaotic regime.
\end{abstract}

\section{Introduction} \label{s: intro}

Can one account for uncertainty in the initial condition of a dynamical system when predicting its future behaviour? This question arises, for example, in astrodynamics, where measurement errors and the chaotic nature of multi-body orbital systems make it a challenge to track space debris or to predict the orbit of hazardous asteroids~\cite{Luo2017AMechanics,Vasile2019}.

If the system's initial state is a random variable with a given probability distribution, the future state distribution can be found by solving the Liouville equation for deterministic dynamics, or the Fokker--Planck equation for stochastic dynamics. Doing so enables one to calculate future expectations of the system's state or functions thereof. Unfortunately, this approach is practical only if the system's state has dimension three or less, otherwise solving the Liouville or Fokker--Planck equations becomes prohibitively expensive.

Many alternatives have of course already been developed. The simplest one is perhaps the Monte Carlo approach, where statistical analysis is performed after the system's dynamics are simulated using a large number of different initial conditions sampled from the given initial distribution. More sophisticated approaches include unscented transforms \cite{Julier2004}, polynomial chaos expansions \cite{Jones2013}, state transition tensor analysis \cite{Park2006}, expansions via differential algebra \cite{ValliTesi}, and modelling via gaussian mixtures \cite{Terejanu2008}. These techniques successfully approximate future expectations, but introduce uncontrolled simplifications (e.g., by considering a finite number of samples or a series expansion) that make it hard to estimate the approximation accuracy. Moreover, to implement these methods one must know the initial probability distribution of the state variables, or at least be able to sample from it. This is impossible if one has access only to limited statistics of the initial state, such as its mean and variance. 

This work addresses these issues by describing an uncertainty propagation framework for deterministic dynamical systems that can handle nonlinear dynamics as well as partial knowledge of the initial state distribution. Precisely, given limited statistics of the initial state, we bound the expected value of an `observable' function at a future time from above and from below by constructing \emph{auxiliary functions} of the state and time variables. Such functions resemble Lyapunov functions but satisfy different constraints, and have already been used in various guises to study nonlinear systems (see, e.g.,
\cite{cherny2014,
    Fantuzzi2016,
    Goluskin2020,
    Korda2014,
    Lasserre2008,
    Streif2013}).
Their key advantage is that the bounds they imply can be optimized by solving a \emph{linear} program over differentiable functions, even though the original dynamics are nonlinear.
For systems governed by polynomial differential equations, moreover, optimal bounds can be approximated with arbitrary accuracy using polynomial optimization if (i) the initial state statistics and the observable of interest are described by polynomials, and (ii) the system evolves in a compact semialgebraic set satisfying a mild technical condition.

Finally, we stress that many of the ideas in this paper have already appeared in the broad literature on polynomial optimization for nonlinear dynamics. In particular, the approach we present here is exactly dual to `occupation measure' relaxations used in~\cite{Streif2014} for parameter identification. All results we present could be derived from those in that work via convex duality. Here, however, we present a direct and much more elementary derivation of the method in the (slightly different) context of uncertainy propagation. 



\section{Problem statement}
\label{ss:deterministic}

We consider deterministic dynamical systems governed by the ordinary differential equation (ODE)
\begin{equation}\label{e:ode}
    \smallddt {\vec{x}}=\vec{f}(t,\vec{x}), \quad \vec{x}(0)=\vec{x}_0.
\end{equation}
Here $\vec{x}(t) \in \mathbb{R}^n$ is the state of the system at time $t$, $\vec{x}_0$ is the state at the initial time (taken as zero without loss of generality), and the vector field  $\vec{f}\ :\ \mathbb{R}\times \mathbb{R}^n \rightarrow \mathbb{R}^n$ is such that solutions are unique and exist for all positive times. The solution of \cref{e:ode} at time $t$ is denoted throughout by $\vec{x}(t;\vec{x}_0)$.

{We will assume that the initial condition $\vec{x}_0$ is a random variable, whose probability distribution $\mu_0$ is supported on a set $\mathcal{X}_0 \subseteq \mathbb{R}^n$ and is known only through expectation bounds on a given vector-valued function $\conFun: \mathbb{R}^n \to \mathbb{R}^m$. In other words, we only know that}
\begin{equation}\label{e:setup:initial-stats}
    \E_{\mu_0}(\conFun)
    := \int_{\mathcal{X}_0}  \conFun(\vec{x}_0) \, \rd\mu_0(\vec{x}_0)
    \leq \vec{c}
\end{equation}
for some fixed vector $\vec{c} \in \mathbb{R}^m$ (the integral and the inequality act element-wise). For example, if one knows bounds on the mean and covariance of $\mu_0$, then one may take $\conFun$ to list all monomials in $\vec{x}_0$ of degree $2$ or less. Of course, {we always have $\E_{\mu_0}(1) = 1$ since} $\mu_0$ is a probability measure.

Given \cref{e:setup:initial-stats}, a time $T>0$, and  an \emph{observable} $\objFun:\mathbb{R}^n \to \mathbb{R}$, we seek bounds on the expected value $\E_{\mu_0}[\objFun(\vec{x}(T; \vec{x}_0))]$ calculated using \emph{any} probability measure $\mu_0$ on $\mathcal{X}_0$ consistent with \cref{e:setup:initial-stats}.
{Precisely, if $\Prob(\mathcal{X}_0)$ is the set of probability measures supported on $\mathcal{X}_0$,} we seek upper bounds on
%
\begin{equation}\label{e:setup:max-expectation}
    p^\star :=
    \sup_{
        \substack{
        \mu_0 \in \Prob(\mathcal{X}_0)\\
        \int_{\mathcal{X}_0}  \conFun \, \rd\mu_0 \leq \vec{c}
    }}
    \E_{\mu_0}[g(\vec{x}(T; \vec{x}_0))].
\end{equation}
{Lower bounds on the infimum of the same expectation under the same constraints can be deduced by negating upper bounds on the supremum of $\E_{\mu_0}[-\objFun(\vec{x}(T;\vec{x}_0)]$.}

If the expectation constraints~\cref{e:setup:initial-stats} are dropped from~\cref{e:setup:max-expectation}, the problem reduces to finding the initial condition $\vec{x}_0 \in \mathcal{X}_0$ that maximizes the value of $\objFun$ at time $T$. This problem and some variations related to safety analysis were studied in~\cite{Fantuzzi2020,Miller2020}.
{On the other hand, one could add constraints on the distribution of the state at intermediate times $t_1<\cdots<t_k$. This case was studied in~\cite{Streif2014} from the dual perspective of occupation measures, and our discussion carries through with straightforward changes. It is also immediate to replace the inequalities in \cref{e:setup:initial-stats} with a mix of equalities and inequalities. We focus on inequalities to ease the presentation.}

{Finally, observe that since we have assumed that solutions to the ODE~\cref{e:ode} exist for all positive times, all trajectories starting from $\mathcal{X}_0$ will remain in some set $\mathcal{X} \subseteq \mathbb{R}^n$ up to time $T$. We will assume that such a set has been fixed, allowing for the choice $\mathcal{X}=\mathbb{R}^n$ if no smaller set is known.}

\section{Upper bounds on maximal expectations}
\label{ss:bounds}

{Having set up the problem, we now derive \emph{a priori} upper bounds on the maximal expectation $p^\star$. Let $\mathbb{R}_+^m$ be the cone of $m$-dimensional vectors with nonnegative entries and let $C^1(D)$ be the space of continuously differentiable functions on a domain~$D$. For every $v \in C^1([0,T]\times \mathcal{X})$ set
\begin{equation}\label{e:Lie-derivative}
    \mathcal{L} v(t,\vec{x}) := \partial_t v(t,\vec{x}) + \vec{f}(t,\vec{x}) \cdot \nabla_{\vec{x}} v(t,\vec{x}).
\end{equation}
We begin with an elementary observation.
}
\begin{proposition}\label{prop:lyapunov-conditions}
    Suppose there exist a scalar $\alpha \in \mathbb{R}$, a vector $\boldsymbol{\beta} \in \mathbb{R}_+^m$, and a function $v \in C^1([0,T]\times \mathcal{X})$ such that
    \begin{subequations}
        \begin{align}
        \label{eqn:line-2}
        \mathcal{L} v(t,\vec{x}) &\leq 0 &&\forall(t,\vec{x}) \in [0,T] \times \mathcal{X},\\
        \objFun(\vec{x}) &\leq v(T,\vec{x})
        &&\forall \vec{x} \in \mathcal{X},
        \label{eqn:line-1}\\
        v(0,\vec{x}) &\leq \alpha + \boldsymbol{\beta} \cdot \conFun(\vec{x})
        &&\forall \vec{x} \in \mathcal{X}_0.
        \label{eqn:line-3}
        \end{align}
    \end{subequations}
    Then, for every $\mu_0\in\Prob(\mathcal{X}_0)$ satisfying~\cref{e:setup:initial-stats} there holds
        \begin{equation}
        \label{e:upper-bound}
        \E_{\mu_0}[\objFun(\vec{x}(T; \vec{x}_0))] \leq \alpha + \boldsymbol{\beta} \cdot \vec{c}.
        \end{equation}
\end{proposition}

\begin{proof}
    {Given $\vec{x}_0 \in \mathcal{X}_0$, let $\vec{x}(t; \vec{x}_0)$ solve the ODE \cref{e:ode}. A straightforward application of the chain rule shows that
    $\mathcal{L} v(t, \vec{x}(t,\vec{x}_0) ) = \smallddt v(t, \vec{x}(t;\vec{x}_0))$, 
    so the function $v(t,\vec{x}(t;\vec{x}_0))$ does not increase over the time interval $[0,T]$ by virtue of \cref{eqn:line-2}.} Inequalities~\cref{eqn:line-1,eqn:line-3} then imply
    $
        \objFun(\vec{x}(T; \vec{x}_0))
        \leq
        v(T, \vec{x}(T; \vec{x}_0))
        \leq
        v(0, \vec{x}_0)
        \leq
        \alpha + \boldsymbol{\beta} \cdot \conFun(\vec{x}_0).
    $
    Taking expectations with respect to the distribution $\mu_0$ of the initial condition $\vec{x}_0$ and {using} \cref{e:setup:initial-stats} yields \cref{e:upper-bound}.
\end{proof}


{Maximizing the left-hand side of \cref{e:upper-bound} over all probability measures $\mu_0$ consistent with \cref{e:setup:initial-stats} shows that $p^\star\leq \alpha+\boldsymbol{\beta}\cdot\vec{c}$. Minimizing this upper bound gives the following result.}

\begin{theorem}
    There holds
    \begin{equation}\label{e:weak-duality}
        p^\star \leq \inf_{
            \substack{
                \alpha \in \mathbb{R},\,
                \boldsymbol{\beta} \in \mathbb{R}_+^m\\
                v \in C^1([0,T]\times \mathcal{X})\\
                \text{s.t. \cref{eqn:line-1,eqn:line-2,eqn:line-3}}
            }
        } \left\{\alpha + \boldsymbol{\beta}\cdot \vec{c} \right\}
        =: d^\star.
    \end{equation}
\end{theorem}

{The minimization problem for $d^\star$ is the Lagrangian dual of the occupation measure relaxation of the maximization defining $p^\star$ in~\cref{e:setup:max-expectation} from~\cite{Streif2014}. The advantage of considering the dual problem is that \emph{any} feasible $v$, $\alpha$ and $\boldsymbol{\beta}$ produces a bound on $p^\star$. In contrast, the occupation measure relaxation must be solved exactly to obtain a bound.}

Inequality \cref{e:weak-duality} expresses a weak duality between the maximal expectation problem \cref{e:setup:max-expectation} and the minimization problem for $d^\star$. This duality is in fact strong, meaning that $d^\star=p^\star$, if the sets $\mathcal{X}_0$ and $\mathcal{X}$ are compact.

\begin{theorem}\label{th:strong-duality}
    If $\mathcal{X}_0$ and $\mathcal{X}$ are compact, $p^\star = d^\star$.
\end{theorem}

This result can be proven by applying either a minimax argument or an abstract strong duality theorem (e.g., \cite[Th.~C.20]{Lasserre2009book}) to the occupation measure relaxation of~\cite{Streif2014}, which evaluates $p^\star$ exactly if $\mathcal{X}_0$ and $\mathcal{X}$ are compact. This statement is not proven explicitly in~\cite{Streif2014}, although the argument is similar to proofs found in \cite{Lasserre2008,Korda2014}. We provide the details in \cref{s:proofs} for the benefit of non-expert readers. First, however, we show how to optimize bounds in practice and showcase the method on examples.

\section{Implementation via polynomial optimization} \label{s: sos}

Computing the upper bound $d^\star$ on the maximal expectation $p^\star$ requires solving the infinite-dimensional linear program in \cref{e:weak-duality}, which is not easy in general. However, upper bounds on $d^\star$ can be computed via semidefinite programming if the ODE vector field $\vec{f}$ is polynomial, the observable function $g$ is also polynomial, and the sets $\mathcal{X}$ and $\mathcal{X}_0$ are defined by a finite number of polynomial inequalities. In this case, one can optimize polynomial auxiliary functions $v$ of fixed degree by strengthening the inequalities \cref{eqn:line-1,eqn:line-2,eqn:line-3} into stronger but tractable sum-of-squares polynomial constraints.

\subsection{Preliminaries}
Let $\mathbb{R}[\vec{x}]$ be the space of polynomials with $\vec{x} \in \mathbb{R}^n$ as the independent variable, and let $\Sigma[\vec{x}]$ be the subset of polynomials that are sums of squares of other polynomials. Polynomials in $\Sigma[\vec{x}]$ are trivially nonnegative on $\mathbb{R}^n$; the converse is generally false except for univariate, quadratic, or bivariate quartic polynomials \cite{Hilbert1888}. To consider only polynomials with total degree less than or equal to $\omega$ we write $\mathbb{R}_\omega[\vec{x}]$ and $\Sigma_{\omega}[\vec{x}]$. Analogous notation (e.g., $\mathbb{R}[t,\vec{x}]$) is used for polynomials that depend on $t$ and $\vec{x}$.

A set $\mathcal{S}\subset\mathbb{R}^n$ is called \emph{semialgebraic} if
\begin{equation*}
    \mathcal{S} := \left\{ \vec{x} \in \mathbb{R}^n: \phi_1(\vec{x}) \geq 0, \, \ldots,\, \phi_p(\vec{x}) \geq 0\right\}
\end{equation*}
for $\phi_1,\ldots,\phi_p \in \mathbb{R}[\vec{x}]$.
The associated \emph{quadratic module} $\mathcal{Q}(\mathcal{S})$ is the set of weighted sum-of-squares polynomials, where the weights are $1$ and the polynomials defining $\mathcal{S}$:
\begin{equation*}
    \mathcal{Q}(\mathcal{S}) := \left\{ \sigma_0 + \sigma_1\phi_1 + \cdots + \sigma_p \phi_p: \sigma_0,\ldots,\sigma_p \in \Sigma[\vec{x}]  \right\}.
\end{equation*}
Its degree-$\omega$ subset is denoted by $\mathcal{Q}_\omega(\mathcal{S}) := \mathcal{Q}(\mathcal{S}) \cap \mathbb{R}_\omega[\vec{x}]$.
Evidently, polynomials in $\mathcal{Q}(\mathcal{S})$ are nonnegative on $\mathcal{S}$. A partial converse result due to Putinar~\cite{Putinar1993} states that every polynomial positive on $\mathcal{S}$ belongs to  $\mathcal{Q}(\mathcal{S})$ if $\mathcal{S}$ satisfies the \emph{Archimedean condition}, i.e., if there exists $r \in \mathbb{R}$ such that
$r^2 - x_1^2 - \cdots - x_n^2 \in \mathcal{Q}(\mathcal{S})$.
This requires $\mathcal{S}$ to be compact and can be ensured for a compact $\mathcal{S}$ by adding to its semialgebraic definition the inequality $r^2 - x_1^2 - \cdots - x_n^2 \geq 0$ with large enough $r$. However, finding such $r$ may be hard.

\subsection{Upper bounds on \texorpdfstring{$d^\star$}{d*} with sum-of-squares constraints}
We now use the definitions given above to compute upper bounds on $d^\star$ using semidefinite programming. Suppose there exists $\omega_0$ such that each component of the vector field $\vec{f}$ of the ODE \cref{e:ode} belongs to $\mathbb{R}_{\omega_0}[t,\vec{x}]$. Suppose also the observable function $g$ is in $\mathbb{R}_{\omega_0}[\vec{x}]$. Finally, suppose the sets
\begin{align*}
    \mathcal{X}_0 &:= \left\{\vec{x} \in \mathbb{R}^n:\; \phi_1(\vec{x})\geq0,\, \ldots,\, \phi_p(\vec{x})\geq 0 \right\}
    \\
    \mathcal{X} &:= \left\{\vec{x} \in \mathbb{R}^n:\; \psi_1(\vec{x})\geq0,\, \ldots,\, \psi_q(\vec{x})\geq 0 \right\}
\end{align*}
are semialgebraic with $\phi_1,\ldots,\phi_p,\psi_1,\ldots,\psi_q\in \mathbb{R}_{\omega_0}[\vec{x}]$. (If $\mathcal{X}= \mathbb{R}^n$ take $q=1$ and $\psi_1(\vec{x})\equiv 1$; the same holds for $\mathcal{X}_0$.) The set $\Omega := [0,T]\times \mathcal{X}$ is then also semialgebraic, as
\begin{equation*}
    \Omega = \left\{ (t,\vec{x}):\; t(T-t)\geq 0, \, \psi_1(\vec{x}) \geq 0,\, \ldots,\,\psi_q(\vec{x})\geq 0 \right\},
\end{equation*}
and it satisfies the Archimedean condition if so does $\mathcal{X}$.

Given these assumptions, consider the minimization problem defining $d^\star$ in \cref{e:weak-duality}. If  $v$ is restricted to be a polynomial of fixed degree $\omega$, then the constraints \cref{eqn:line-1,eqn:line-2,eqn:line-3} can be rearranged as finite-degree polynomial nonnegativity conditions on the sets $\Omega$, $\mathcal{X}$ and $\mathcal{X}_0$, respectively. These are intractable in general, but can be strengthened into the sufficient (weighted) sum-of-squares constraints
\begin{subequations}
        \begin{align}
        \label{eqn:sos-2}
        -\mathcal{L} v &\in \mathcal{Q}_\omega(\Omega),\\
        \label{eqn:sos-1}
        v(T,\cdot) - \objFun &\in\mathcal{Q}_\omega(\mathcal{X}),
        \\
        \label{eqn:sos-3}
        \alpha + \boldsymbol{\beta} \cdot \conFun - v(0,\cdot) &\in \mathcal{Q}_\omega(\mathcal{X}_0).
    \end{align}
\end{subequations}
We therefore conclude the following result.
\begin{proposition}
    Suppose $\vec{f}$, $g$, $\mathcal{X}$ and $\mathcal{X}_0$ satisfy the assumptions at the start of this subsection. For every integer $\omega$,
    \begin{equation}\label{e:sos-bound}
        p^\star \leq d^\star \leq \inf_{
                \substack{
                    \alpha \in \mathbb{R},\,
                    \boldsymbol{\beta} \in \mathbb{R}_+^m\\
                    v \in \mathbb{R}_\omega[t,\vec{x}]\\
                    \text{s.t. \cref{eqn:sos-1,eqn:sos-2,eqn:sos-3}}
                }
            } \left\{\alpha + \boldsymbol{\beta}\cdot \vec{c} \right\}
            =: d_\omega^\star.
    \end{equation}
\end{proposition}

Optimization problems with weighted SOS constraints, such as \cref{e:sos-bound}, can be recast as semidefinite programs~\cite{Laurent2009,Lasserre2015,Parrilo2013book}. In principle, therefore, each upper bound $d_\omega^\star$ can be calculated via semidefinite programming. Barring issues with numerical conditioning, current software can handle problems where the polynomial degree $\omega$ and the dimension $n$ of the state variable $\vec{x}$ 
satisfy $\binom{n+\lfloor\omega/2\rfloor}{\lfloor\omega/2\rfloor}\leq 1000$ or so.

\subsection{Convergence in the compact case}
Since $\mathbb{R}_\omega[t,\vec{x}] \subset \mathbb{R}_{\omega+1}[t,\vec{x}]$ and $\mathcal{Q}_\omega(\cdot) \subset \mathcal{Q}_{\omega+1}(\cdot)$, the upper bounds $d_\omega^\star$ form a nonincreasing sequence as $\omega$ is raised. This sequence is bounded below by $d^\star$, so it has a limit $d_\infty$, but it is not known if $d_\infty=d^\star$ in general. One case in which this equality holds is when the sets $\mathcal{X}$ and $\mathcal{X}_0$ are compact and satisfy the Archimedean condition. Combined with \cref{th:strong-duality}, we obtain our last theoretical result.

\begin{theorem}\label{th:sos-convergence-compact}
    Suppose that $\mathcal{X}$ and $\mathcal{X}_0$ satisfy the Archimedean condition. Then, $p^\star = d^\star = \lim_{\omega \to \infty} d_\omega^\star$.
\end{theorem}

The proof follows standard arguments (e.g., \cite{Fantuzzi2020,Korda2014,Lasserre2008}), so we omit it. Briefly, given any ${\epsilon>0}$, the compactness of $\mathcal{X}_0$ and $\mathcal{X}$ (hence, of $\Omega$) allows one to replace any $v$ satisfying \cref{eqn:line-1,eqn:line-2,eqn:line-3} with a polynomial $\tilde{v}$ satisfying \cref{eqn:line-1,eqn:line-2,eqn:line-3} with strict inequality at the expense of worsening the corresponding bound by no more than $\epsilon$.
Then, by Putinar's theorem, $\tilde{v}$ satisfies also \cref{eqn:sos-1,eqn:sos-2,eqn:sos-3} for large enough $\omega$. One can thus find $\omega$ such that $d_\omega^\star \leq d^\star + \varepsilon$, establishing \cref{th:sos-convergence-compact}.
\section{Examples} \label{s: cases}

To illustrate the uncertainty propagation framework discussed above, we now bound the expected time-$T$ state of the van der Pol oscillator and of the chaotic Lorenz system when the initial state has a prescribed mean $\boldsymbol{\mu}$ and covariance matrix $\Sigma$. In both examples we take $\mathcal{X}_0=\mathcal{X}=\mathbb{R}^n$. All computations use YALMIP \cite{lofberg2004yalmip,Lofberg2009sos} to recast the polynomial optimization problem for $d^\star_\omega$ into a semidefinite program, which is solved using MOSEK~\cite{mosek}.
In each example, we compare our bounds to expected state values approximated by simulating $10^6$ trajectories whose initial conditions is drawn from either the normal or the uniform distribution with mean $\boldsymbol{\mu}$ and covariance matrix $\Sigma$. (The ensamble mean and covariance matrices of the drawn set of initial conditions agree with the nominal values to within two decimal places.)

\begin{table}[t]
    \vskip5pt
    \caption{Lower and upper bounds (LB and UB) on the expected time-$T$ state of the van der Pol oscillator, computed with degree-$\omega$ polynomial $v$ and $\mu_0$ with mean and covariance matrix in~\cref{e:vdp-data}.}
    \centering
    \begin{tabular}{cc c rr c rr}
    \toprule
    &&&\multicolumn{2}{c}{$\E_{\mu_0}[x_1(T;\vec{x}_0)]$}&& \multicolumn{2}{c}{$\E_{\mu_0}[x_2(T;\vec{x}_0)]$}\\[0.5ex]
    $T$ & $\omega$ &&
    \multicolumn{1}{c}{LB} & 
    \multicolumn{1}{c}{UB} &&
    \multicolumn{1}{c}{LB} & 
    \multicolumn{1}{c}{UB} \\
    [0.75ex]
    %
    1 &  8 &&  0.2857 &  0.2917 &&  0.1087 &  0.1213 \\
    1 & 12 &&  0.2858 &  0.2916 &&  0.1087 &  0.1209 \\
    1 & 16 &&  0.2858 &  0.2916 &&  0.1087 &  0.1209 \\
    [0.5ex]
    %
    2 &  8 &&  0.2397 &  0.2621 && -0.2147 & -0.1832 \\
    2 & 12 &&  0.2411 &  0.2603 && -0.2082 & -0.1861 \\
    2 & 16 &&  0.2412 &  0.2603 && -0.2073 & -0.1862 \\
    [0.5ex]
    %
    3 &  8 && -0.1744 & -0.1103 && -0.6353 & -0.5251 \\
    3 & 12 && -0.1580 & -0.1269 && -0.6284 & -0.5700 \\
    3 & 16 && -0.1556 & -0.1281 && -0.6271 & -0.5781 \\
    [0.5ex]
    %
    4 &  8 && -0.5931 & -0.4945 && -0.1891 &  0.0213 \\
    4 & 12 && -0.5907 & -0.5538 && -0.0953 & -0.0127 \\
    4 & 16 && -0.5904 & -0.5615 && -0.0781 & -0.0200 \\
    [0.5ex]
    %
    5 &  8 && -0.4776 & -0.2698 &&  0.1056 &  0.4339 \\
    5 & 12 && -0.4457 & -0.3877 &&  0.2343 &  0.3251 \\
    5 & 16 && -0.4431 & -0.3931 &&  0.2508 &  0.3182 \\
    \bottomrule
    \label{tab:vdp-res}
    \end{tabular}
\end{table}

\begin{figure}[t]
    \centering 
    \includegraphics[width=0.99\linewidth, trim=0cm 0.25cm 0cm 0.8cm]{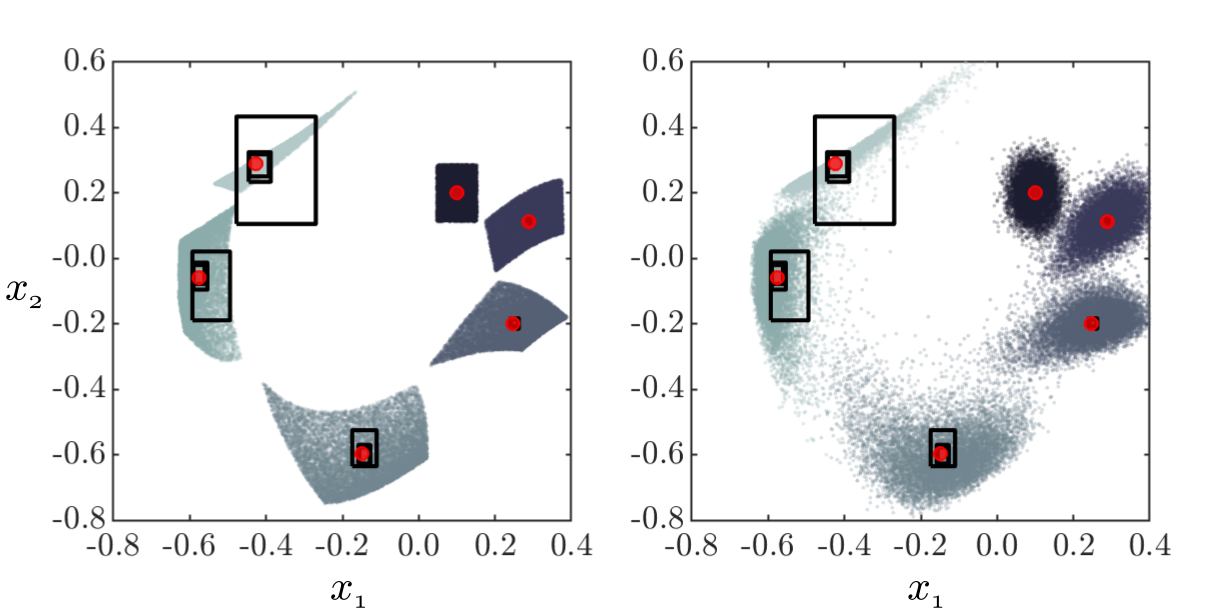}
    \caption{
        Expected state (\red{$\bullet$}) for the van der Pol oscillator at times
        $T=0$ (\dpone{$\bullet$}), 
        $1$ (\dptwo{$\bullet$}),
        $2$ (\dpthree{$\bullet$}),
        $3$ (\dpfour{$\bullet$}),
        $4$ (\dpfive{$\bullet$}) and 
        $5$ (\dpsix{$\bullet$}), approximated using $10^6$ initial conditions  drawn from a uniform (left) and a normal (right) distributions with the mean and covariance matrix in \cref{e:vdp-data}. Also shown are the state distribution at each time $T$. Boxes show the expectation bounds from \cref{tab:vdp-res}.
    \label{fig:vdp-res}}
\end{figure}

\subsection{The van der Pol oscillator}\label{ss:examples:vdp}
Let us consider the van der Pol oscillator
\begin{equation}\label{e:vdp}
    \ddt{} \begin{pmatrix}
    x_1\\
    x_2
    \end{pmatrix} =
    \begin{pmatrix}
    x_1\\
    (1-9x_1^2)x_2-x_1
    \end{pmatrix},
\end{equation}
which for numerical reasons was rescaled so its limit cycle lies in the unit box $[-1,1]^2$. We seek upper and lower bounds on the expected value of each state variable $x_i$ at time $T$ for initial state distributions $\mu_0$ with mean and covariance matrix
\begin{equation}
    \boldsymbol{\mu}= \begin{pmatrix} 0.1 \\ 0.2\end{pmatrix}
    \quad\text{and}\quad
    \Sigma =
    \begin{pmatrix}
    0.03^2 & 0 \\
    0 &  0.05^2
    \end{pmatrix}.
    \label{e:vdp-data}
\end{equation}
Upper bounds on $\E_{\mu_0}[x_i(T;\vec{x}_0)]$ ($i=1$ or $2$) can be computed as discussed in \cref{ss:bounds,s: sos} upon setting $g(\vec{x})=x_i$, $\conFun(x)=(x, y, x^2, xy, y^2)$, and $\vec{c} = (\mu_1, \mu_2, \mu_1^2+\Sigma_{11}, \mu_1\mu_2+\Sigma_{12}, \mu_2^2+\Sigma_{22})$; the only change required to handle the constraint $\E_{\mu_0}(\vec{h})=\vec{c}$ instead of $\E_{\mu_0}(\vec{h})\leq \vec{c}$ is that one drops the nonnegativity constraints on the entries of $\boldsymbol{\beta}$ in \cref{e:weak-duality,e:sos-bound}. Lower bounds on $\E_{\mu_0}[x_i(T;\vec{x}_0)]$ are deduced by negating upper bounds on $\E_{\mu_0}[-x_i(T;\vec{x}_0)]$.

\Cref{tab:vdp-res} lists the bounds computed by solving \cref{e:sos-bound} for $T\in\{1,2,3,4,5\}$ and $\omega \in \{8,12,16\}$. \Cref{fig:vdp-res} compares these bounds to the expected state values approximated using the simulation setup described at the start of this section. As expected, the bounds improve as $\omega$ is raised and always bracket the simulation results. Higher $\omega$ is needed to obtained well-converged bounds at larger times, which is not surprising since nonlinear uncertainty propagation with larger time horizons is expected to be harder. \Cref{fig:vdp-res} also suggest that our bounds become sharp as $\omega \to \infty$, but we cannot guarantee this because \cref{th:sos-convergence-compact} does not apply to our computations where $\mathcal{X}=\mathcal{X}_0=\mathbb{R}^2$.

\subsection{Lorenz system}\label{ss:examples:lorenz}
As our second example we consider the Lorenz system with the standard parameter values for chaotic behaviour,
\begin{equation}\label{e:lorenz}
    \ddt{}
    \begin{pmatrix}
    x_1\\
    x_2\\
    x_3\\
    \end{pmatrix} =
    \begin{pmatrix}
    10 (x_2-x_1)\\
    x_1(28 - x_3) - x_2\\
    x_1x_2- \tfrac83 x_3 \\
    \end{pmatrix}.
\end{equation}
We seek bounds on the time-$T$ expectation of each $x_i$ for initial state distributions with the mean and covariance matrix
\begin{equation}\label{e:lorenz-data}
        \boldsymbol{\mu}=\begin{pmatrix}
        2.254\\
        4.029\\
        10.646\end{pmatrix}
        \;\text{and}\;\,
        {\Sigma} =
        \begin{pmatrix}
        0.314 & 0 & 0 \\
        0 &  0.706 & 0 \\
        0 & 0 & 0.314
        \end{pmatrix}\!.
\end{equation}
(The mean $\boldsymbol{\mu}$ is a point near the chaotic attractor.) Results for $T=\{0.2,0.4,0.6,0.8\}$, along with approximate expectations from simulations, are shown in \cref{tab:lorenz,fig:lorenz} and are analogous to those obtained in \cref{ss:examples:vdp}. The bounds are also relatively robust to the initial mean and covariance: replacing \cref{e:lorenz-data} with the
initial mean and covariance of the simulated ensemble of
trajectories changed the bounds for $\omega=14$ by less than $4\%$. Our approach to uncertainty propagation thus works well also for chaotic systems. 
\section{Proofs}
\label{s:proofs}

To make this paper self-contained and accessible to non-experts, we now give a detailed proof of \cref{th:strong-duality}, which to the best of our knowledge is absent from the literature. The ideas are similar to proofs in \cite{Lasserre2008,Korda2014} and many other works presenting polynomial optimization frameworks for dynamical system analysis. All arguments except the minimax steps in \cref{ss:strong-duality-proof} below were very briefly outlined in \cite{Streif2014}. As in \cref{s: sos}, we ease the notation by setting $\Omega := [0,T]\times \mathcal{X}.$

\subsection{A measure-theoretic reformulation}
We begin by recasting the maximization problem in \cref{e:setup:max-expectation} as a problem over measures. We do this by introducing the state distribution at time $t$, which is denoted by $\mu_t$ and satisfies
\begin{equation}\label{e:pushforward-conditions}
    \int_{\mathbb{R}^n} \Psi(\vec{x}) \,\rd\mu_t(\vec{x})
    =
    \int_{\mathcal{X}_0} \Psi(\vec{x}(t;\vec{x}_0)) \,\rd\mu_0(\vec{x}_0)
\end{equation}
for every continuous function $\Psi:\mathbb{R}^n \to \mathbb{R}$. The integral on the left-hand side can be restricted to the support of $\mu_t$, which is the time-$t$ image of $\mathcal{X}_0$ under the system dynamics, and therefore to the set $\mathcal{X}$ in which all trajectories starting from $\mathcal{X}_0$ remain for times $t \in [0,T]$.

It is well known (see, e.g., \cite{Ambrosio2008}) that the family of probability measures $\{\mu_t\}_{t \in [0,T]}$ satisfies the Liouville equation
\begin{equation}\label{e:setup:liouville}
    \partial_t \mu_t + \nabla_{\vec{x}} \cdot (\vec{f} \mu_t) = 0
\end{equation}
with initial condition $\mu_0$. The equation holds in the sense of distributions, meaning that
\begin{equation}\label{e:setup:weak-liouville}
    \int_{0}^T \int_{\mathcal{X}} 
    \mathcal{L}v(t,\vec{x}) \,\rd\mu_t(\vec{x}) \rd t = 0
\end{equation}
for every smooth function $v$ compactly supported in the interior of $[0,T] \times \mathcal{X}$. Recall from \cref{e:Lie-derivative} that $\mathcal{L}v = \partial_t v + \vec{f} \cdot \nabla_{\vec{x}} v$.

In fact, the equation $\mathcal{L}v=\ddt v$ along solutions of the ODE \cref{e:ode} can be integrated first in time and then against the the probability measure $\mu_0$ to conclude that
\begin{multline}\label{e:measures-identity}
    \int_{0}^T \int_{\mathcal{X}} 
    \mathcal{L}v(t,\vec{x}) \,\rd\mu_t(\vec{x}) \rd t
    =\\
    \int_{\mathcal{X}} v(T,\vec{x}) \,\rd\mu_T(\vec{x})
    - \int_{\mathcal{X}_0} v(0,\vec{x}) \,\rd\mu_0(\vec{x})
\end{multline}
for every continuously differentiable $v$ for which all integrals exist. Such identities completely characterize solutions to the Liouville equation as per the following result, where $T\Prob(\Omega)$ is the sets of nonnegative measures with mass $T$ on $\Omega$.

\begin{table}[t]
    \vskip5pt
    \caption{
    Lower and upper bounds (LB and UB) on the expected time-$T$ state of the Lorenz system \cref{e:lorenz}, computed with degree-$\omega$ polynomial $v$ and $\mu_0$ with mean and covariance matrix in \cref{e:lorenz-data}.
    }
    \centering
    \begin{tabular}{cc rr rr rr}
    \toprule
    &&\multicolumn{2}{c}{$\E_{\mu_0}[x_1(T;\vec{x}_0)]$}
    & \multicolumn{2}{c}{$\E_{\mu_0}[x_2(T;\vec{x}_0)]$}
    & \multicolumn{2}{c}{$\E_{\mu_0}[x_3(T;\vec{x}_0)]$}
    \\[0.5ex]
    $T$ & $\omega$ &
    \multicolumn{1}{c}{LB} & 
    \multicolumn{1}{c}{UB} &
    \multicolumn{1}{c}{LB} & 
    \multicolumn{1}{c}{UB} &
    \multicolumn{1}{c}{LB} & 
    \multicolumn{1}{c}{UB} \\
    [0.75ex]
    0.2 & 10 &  11.472 &  11.640 &  19.099 &  19.587 &  18.894 &  19.389 \\
    0.2 & 12 &  11.482 &  11.639 &  19.158 &  19.585 &  18.905 &  19.343 \\
    0.2 & 14 &  11.482 &  11.639 &  19.161 &  19.585 &  18.905 &  19.340 \\
    [0.5ex]
    0.4 & 10 &  7.110 &  8.321 & -2.820 & -0.816 &  34.456 &  36.382 \\
    0.4 & 12 &  7.268 &  8.275 & -2.700 & -1.031 &  34.746 &  36.345 \\
    0.4 & 14 &  7.301 &  8.261 & -2.666 & -1.107 &  34.840 &  36.338 \\
    [0.5ex]
    0.6 & 10 & -2.536 & -1.051 & -4.485 & -2.375 &  19.708 &  21.655 \\
    0.6 & 12 & -2.469 & -1.190 & -4.323 & -2.602 &  19.831 &  21.417 \\
    0.6 & 14 & -2.429 & -1.190 & -4.350 & -2.598 &  19.887 &  21.378 \\
    [0.5ex]
    0.8 & 10 & -7.811 & -2.666 & -12.294 & -0.635 &  14.936 &  24.344 \\
    0.8 & 12 & -7.495 & -4.520 & -11.928 & -6.017 &  15.734 &  19.740 \\
    0.8 & 14 & -7.353 & -5.290 & -11.679 & -8.556 &  15.871 &  18.261 \\
    \bottomrule
    \label{tab:lorenz}
    \end{tabular}
\end{table}

\begin{figure}
    \centering
    \includegraphics[width=0.99\linewidth, trim=0cm 0.3cm 0cm 0.85cm]{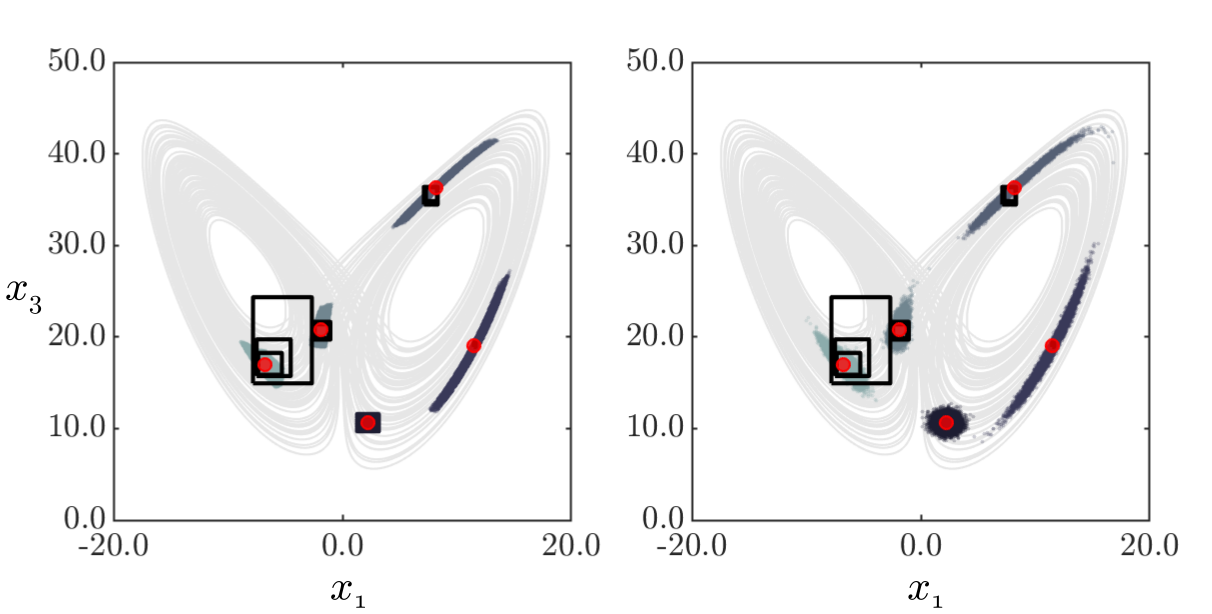}
    \caption{Expected state (\red{$\bullet$}) for the Lorenz system at times
    $T=0$ (\dpone{$\bullet$}),
    $0.2$ (\dptwo{$\bullet$}),
    $0.4$ (\dpthree{$\bullet$}),
    $0.6$ (\dpfour{$\bullet$}) and
    $0.8$ (\dpfive{$\bullet$}), approximated using $10^6$ initial conditions with a uniform (left) and a normal (right) distribution with mean and covariance matrix in \cref{e:lorenz-data}. The attractor and state distribution at each time $T$ are also shown. Boxes show the bounds from \cref{tab:lorenz}.
   \label{fig:lorenz}}
\end{figure}

\begin{lemma}\label{lemma:liouville-characterization}
    If $\lambda \in \Prob(\mathcal{X}_0)$, $\nu\in \Prob(\mathcal{X})$, $\mu \in T\Prob(\Omega)$ satisfy
    \begin{multline}\label{e:relaxed-measures-identity}
    \int_{\Omega} \mathcal{L}v(t,\vec{x}) \,\rd\mu(t,\vec{x})
    =\\
    \int_{\mathcal{X}} v(T,\vec{x}) \,\rd\nu(\vec{x})
    - \int_{\mathcal{X}_0} v(0,\vec{x}) \,\rd\lambda(\vec{x})
    \end{multline}
    for all $v \in C^1(\Omega)$, then $\mu$ admits the disintegration
    $\mu(\rd t,\rd \vec{x}) = \mu_t(\rd\vec{x}) \rd t$
    where $\mu_t:[0,T]\to \mathcal{X}$ solves the Liouville equation \cref{e:setup:liouville}. Moreover, $\lambda=\mu_0$ and $\nu = \mu_T$.
\end{lemma}

\begin{proof}
    Setting $v(t,\vec{x}) = t^\alpha$ for every $\alpha \in \mathbb{N}$ shows that the $t$-marginal of $\mu$ is the Lebesgue measure on $[0,T]$. Thus, by a disintegration theorem for measures~\cite[Theorem~4.4]{Rindler2018}, for almost every $t$ there exists $\mu_t \in \Prob(\mathcal{X})$ such that
    \begin{equation*}
    \int_\Omega \mathcal{L}v(t,\vec{x}) \,\rd\mu(t,\vec{x})
    =
    \int_{0}^T \int_{\mathcal{X}} \mathcal{L}v(t,\vec{x}) \,\rd\mu_t(\vec{x}) \rd t.
    \end{equation*}
    Substituting this identity in \cref{e:relaxed-measures-identity} and taking $v$ to be smooth and compactly supported in the interior of $\Omega = [0,T] \times \mathcal{X}$ gives \cref{e:setup:weak-liouville}, so $\mu_t$ solves the Liouville equation. This means $\mu_t$ satisfies \cref{e:pushforward-conditions} \cite{Ambrosio2008}, hence also \cref{e:measures-identity} for every $v \in C^1(\Omega)$. Combining this with \cref{e:relaxed-measures-identity} yields
    \begin{multline*}
    \int_{\mathcal{X}} v(T,\vec{x}) \,\rd\mu_T(\vec{x})
    - \int_{\mathcal{X}_0} v(0,\vec{x}) \,\rd\mu_{0}(\vec{x})
    =\\
    \int_{\mathcal{X}} v(T,\vec{x}) \,\rd\nu(\vec{x})
    - \int_{\mathcal{X}_0} v(0,\vec{x}) \,\rd\lambda(\vec{x}).
    \end{multline*}
    Choosing $v(t,\vec{x})=\xi(t)\eta(\vec{x})$ with $\xi(T)=0$ (resp. $\xi(0)=0$) and $\eta$ arbitrary shows that $\lambda= \mu_0$ (resp. $\nu = \mu_T$).
\end{proof}

\Cref{lemma:liouville-characterization} enables us to evaluate the maximal expectation in \cref{e:setup:max-expectation} via the solution of a linear program over measures, which is the occupation measure relaxation described in~\cite{Streif2014}.

\begin{proposition}[\cite{Streif2014}]\label{prop:max-expectation-measures}
    There holds
    \begin{equation}\label{e:setup:max-expectation-measures}
    p^\star = \sup_{
        \substack{
        (\mu_0,\nu,\mu)\in \Prob(\mathcal{X}_0)\times \Prob(\mathcal{X}) \times T\Prob(\Omega)\\
        \text{s.t. \cref{e:relaxed-measures-identity} and }
        \int_{\mathcal{X}_0} \conFun \,\rd\mu_0 \leq \vec{c}
        }}
        \int_{\mathcal{X}} \objFun(\vec{x}) \, \rd\nu(\vec{x}).
    \end{equation}
\end{proposition}
\begin{proof}
    Use \cref{lemma:liouville-characterization} to disintegrate $\mu$. Its density $\mu_t$ solves the Liouville equation, so $\nu=\mu_T$ is the time-$T$ image of $\mu_0$.
    Then, \cref{e:pushforward-conditions} shows the equivalence of \cref{e:setup:max-expectation-measures} and \cref{e:setup:max-expectation}.
\end{proof}

\subsection{Proof of \texorpdfstring{\cref{th:strong-duality}}{theorem \ref{th:strong-duality}}}\label{ss:strong-duality-proof}
We now turn to the proof of \cref{th:strong-duality}. We apply a minimax argument to the maximization in~\cref{e:setup:max-expectation-measures} to conclude that $p^*=d^*$. Specifically, let
\begin{align*}
    \Lambda(\mu_0,\mu,\nu; v,\boldsymbol{\beta}) :=
    \boldsymbol{\beta}\cdot\vec{c}
    +
    \int_{\mathcal{X}_0} \left[v(0,\vec{x}) - \boldsymbol{\beta} \cdot \vec{h}(\vec{x})\right] \rd\mu_0(\vec{x})&
    \\
    + \int_{\mathcal{X}} \left[g(\vec{x}) - v(T,\vec{x}) \right]\rd\nu(\vec{x})&
    \\
    + \int_{\Omega} \mathcal{L}v(t,\vec{x}) \, \rd\mu(t,\vec{x})&
\end{align*}
be the Lagrangian corresponding to the maximization in~\cref{e:setup:max-expectation-measures}, where $v \in C^1(\Omega)$ and $\beta \in \mathbb{R}^m$ are Lagrange multipliers for the constraints of~\cref{e:setup:max-expectation-measures}. Set 
$\mathbb{X} = \Prob(\mathcal{X}_0) \times T\Prob(\Omega) \times \Prob(\mathcal{X})$ and $\mathbb{Y} = \mathbb{R}^m_+ \times C^1(\Omega)$.
We claim that
\begin{subequations}
    \begin{align}
    \label{e:strong-duality:step1}
    p^\star
    &=
    \adjustlimits
    \sup_{
        (\mu_0,\lambda,\nu)\in\mathbb{X}
    }
    \inf_{
        (\boldsymbol{\beta},v)\in\mathbb{Y}
    }
    \;
    \Lambda(\mu_0,\lambda,\nu; v,\boldsymbol{\beta})
    \\
    \label{e:strong-duality:step2}
    &=
    \adjustlimits
    \inf_{
        (\boldsymbol{\beta},v)\in\mathbb{Y}
    }
    \sup_{
        (\mu_0,\lambda,\nu)\in\mathbb{X}
    }
    \;
    \Lambda(\mu_0,\lambda,\nu; v,\boldsymbol{\beta})
    \\
    \label{e:strong-duality:step3}
    &= d^\star.
\end{align}
\end{subequations}

Equality~\cref{e:strong-duality:step1} holds since the inf on the right-hand side is $-\infty$ unless $\lambda$, $\nu$ and $\mu$ satisfy the constraints of~\cref{e:setup:max-expectation-measures}.

For the equality in~\cref{e:strong-duality:step2}, recall that $\mathcal{X}$ and $\mathcal{X}_0$ are compact. Then, the set $\mathbb{X}$ equipped with the product weak-star topology is a compact convex subset of a linear topological space. Observe also that $\mathbb{Y}$ with the usual product topology is a convex subset of a linear space.
Finally, note that the Lagrangian $\Lambda$ is linear and continuous on $\mathbb{X} \times \mathbb{Y}$ for the stated topologies. We have thus verified all assumptions of a minimax theorem due to Sion \cite{Sion1958}, which guarantees that the order of the inf and the sup in~\cref{e:strong-duality:step1} is irrelevant.

Finally, we turn to the equality in~\cref{e:strong-duality:step3}. The optimal $\mu_0$, $\nu$ and $\mu$ in~\cref{e:strong-duality:step2} are Dirac measures at the maximizers of $v(0,\cdot) -\boldsymbol{\beta} \cdot \vec{h}(\cdot)$, $g(\cdot) - v(T,\cdot)$ and $\mathcal{L}v$, respectively. Thus,
\begin{align*}
    p^\star =
    \inf_{\substack{
        \boldsymbol{\beta} \in \mathbb{R}_+^m\\
        v \in C^1(\Omega)
    }}
    \big\{
        \boldsymbol{\beta}\cdot\vec{c} +
        \max_{\vec{x} \in \mathcal{X}_0} \left[v(0,\vec{x}) - \boldsymbol{\beta} \cdot \vec{h}(\vec{x})\right]&
        \\[-2.5ex]
        + \max_{\vec{x} \in \mathcal{X}} \left[g(\vec{x}) - v(T,\vec{x})\right]&
        \\
        + T \max_{(t,\vec{x}) \in \Omega} \mathcal{L}v(t,\vec{x})&
    \big\}.
\end{align*}
The right-hand infimum is unchanged if $v$ is replaced with
\begin{align*}
    v'(t,\vec{x}) := v(t,\vec{x})
    &+ \max_{\vec{x} \in \mathcal{X}} \left[g(\vec{x}) - v(T,\vec{x})\right]
    \\
    &+ (T-t) \max_{(t,\vec{x}) \in \Omega} \mathcal{L}v(t,\vec{x}),
\end{align*}
which satisfies \cref{eqn:line-1,eqn:line-2}. Consequently,
\begin{equation*}
    p^\star = \inf_{\substack{
        \boldsymbol{\beta} \in \mathbb{R}_+^m,\,
        v \in C^1(\Omega)\\
        \text{s.t. \cref{eqn:line-1,eqn:line-2}}
    }}
        \big\{
        \boldsymbol{\beta}\cdot\vec{c} + \max_{\vec{x} \in \mathcal{X}_0} \left[v(0,\vec{x}) - \boldsymbol{\beta} \cdot \vec{h}(\vec{x})\right]
        \big\}.
\end{equation*}
The right-hand side is clearly equal to $d^\star$.

\section{Conclusion} \label{s: conclusione}

We used convex optimization to propagate uncertainty in the initial condition of nonlinear ODEs. Our approach bounds the expectation of a system observable $g$ by searching for auxiliary functions and is dual to the occupation measure relaxations developed for parameter identification in \cite{Streif2014}. Expectation bounds can be estimated computationally by solving semidefinite programs if the ODE vector field and the observable are polynomials. For systems evolving in compact sets, the optimal upper (lower) bound evaluates the maximum (minimum) expectation consistent with the initial uncertainty and can be approximated numerically to any accuracy.

The method we described can be extended in multiple directions. First, computations can be carried out for non-polynomial ODEs and observables provided the constraints of problem \cref{e:sos-bound} can be recast as polynomial inequalities. Examples are rational functions, sinusoidal functions and (as in orbital mechanics) rational powers of
polynomials. Second, one can apply our method to stochastic processes by replacing the operator $\mathcal{L}$ with the generator of the stochastic process. In this case, however, care must be taken to check that the convergence results we presented carry over as expected.
We believe that this flexibility makes the convex optimization framework we described a valuable tool for uncertainty propagation in low-dimensional nonlinear systems. 
Confirming this on examples of practical relevance, however, remains a challenge.

\bibliographystyle{./bib-styles/IEEEtranS}
\bibliography{reflist}
\end{document}